\documentclass[11pt]{amsart}
\usepackage{graphicx}

\usepackage[square,sort&compress,comma,numbers]{natbib}
\usepackage{nicefrac,xcolor,upref,version}
\usepackage{amscd,amsthm,amsmath,amssymb,amsfonts}

\usepackage{mathrsfs}

\usepackage[colorlinks=true,citecolor=cyan,backref=page]{hyperref}

\usepackage{shuffle}
\usepackage{ulem}

\usepackage{enumerate}

\newtheorem{theorem}{Theorem}[section]
\newtheorem{lemma}[theorem]{Lemma} 
\newtheorem{proposition}[theorem]{Proposition} 
\newtheorem{definition}[theorem]{Definition} 
\newtheorem{problem}[theorem]{Problem} 
\newtheorem{corollary}[theorem]{Corollary}

\numberwithin{equation}{section}

\def\eps{{\varepsilon}}

\def\bft{{\bf t}}

\def\cL{{\mathcal L}}

\def\F{{\mathbb F}}
\def\K{{\mathbb K}}

\def\tiM{{\widetilde M}}

\def\tiP{{\widetilde P}}
\def\tiQ{{\widetilde Q}}
\def\tiR{{\widetilde R}}
\def\tiU{{\widetilde U}}
\def\tiV{{\widetilde V}}
\def\tiW{{\widetilde W}}
\def\tiZ{{\widetilde Z}}

\def\cA{{\mathcal A}}
\def\cL{{\mathcal L}}
\def\cM{{\mathcal M}} 

\def\cN{{\mathcal N}}

\allowdisplaybreaks

\begin{document}

\title[Thue--Morse continued fractions in characteristic $2$]{The Thue--Morse continued fractions in characteristic $2$  are algebraic}

\author{Yann Bugeaud}
\address{Universit\'e de Strasbourg, Math\'ematiques,
7, rue Ren\'e Descartes, 67084 Strasbourg  (France)}
\address{Institut universitaire de France}
\email{bugeaud@math.unistra.fr}

\author{Guo-Niu Han}
\address{I.R.M.A., UMR 7501, Universit\'e de Strasbourg
et CNRS, 7 rue Ren\'e Descartes, 67084 Strasbourg Cedex, France}
\email{guoniu.han@unistra.fr}

\begin{abstract}
Let $a, b$ be distinct, non-constant polynomials in $\F_2 [z]$. 
Let $\xi_{a, b}$ be the power series in $\F_2((z^{-1}))$ whose sequence 
of partial quotients is the Thue--Morse sequence over $\{a, b\}$. 
We establish that $\xi_{a,b}$ is algebraic of degree~$4$. 
\end{abstract}

    %Thue--Morse sequence, t

\maketitle

%\tableofcontents

\section{Introduction}\label{intro}

An infinite sequence $(a_n)_{n \ge 0}$ over a finite alphabet $\cA$ is {\it $k$-automatic} 
for an integer $k \ge 2$ if it can be generated by a finite automaton which reads
the representation in base $k$ of a non-negative integer $n$ from right to left 
and outputs an element $a_n$ in $\cA$. It is called {\it automatic} if it is $k$-automatic for some integer 
$k \ge 2$. 
In 1968, Cobham \cite{Cob68} conjectured that, for any integer $b \ge 2$, the
base-$b$ expansion of an irrational real number never forms an automatic sequence. 
%$\xi$ must be transcendental. 
%many authors have studied the 
%Diophantine nature of real numbers whose expansion in some integer base or whose 
%continued fraction expansion is given by an automatic sequence. 
This was confirmed in 2007 by Adamczewski and Bugeaud \cite{AdBu07a} (see 
\cite{Ph15,AdFa17} for alternative proofs), who also established that if the Hensel 
expansion of an irrational $p$-adic number $\xi$ is an automatic sequence, then $\xi$ is transcendental. 
%proved that, for any integer $b \ge 2$, if the
%base-$b$ expansion of an irrational real number $\xi$ forms an automatic sequence, then 
%$\xi$ must be transcendental. 
In a similar spirit, Bugeaud \cite{Bu13} proved that the sequence of partial 
quotients of an algebraic real number of degree at least $3$ never forms an automatic sequence. 

%The proofd rest on the SST.

Analogous questions can be asked for power series over a finite field, but their answers 
are different.
Throughout the paper, 
we let $q$ be a power of a prime number $p$ and $\F_q$ denote the field with $q$ elements. 
In 1979 Christol \cite{Chr79,CKMFR} 
established that a power series in $\F_q((z^{-1}))$ is algebraic over $\F_q(z)$ 
if and only if the sequence of its coefficients is $q$-automatic (or, equivalently, 
is $p$-automatic). 

%The situation is the opposite 
%for real numbers and for $p$-adic numbers. 

In analogy with the continued fraction algorithm for real numbers, there is a 
well-studied continued fraction algorithm for power series in $\F_q((z^{-1}))$, 
the partial quotients being nonconstant polynomials in $\F_q[z]$. 
In both settings, eventually periodic expansions correspond to quadratic elements, but 
much more is known on the continued fraction expansion of algebraic power series 
than on that of algebraic real numbers. 
In 1949 Mahler \cite{Mah49} established that, for $p \ge 3$, the root 
$$
z^{-1} + z^{-p} + z^{-p^2} + \ldots 
$$
in $\F_q((z^{-1}))$ of the polynomial $z X^p - z X + 1$ 
%shows 
%that there exist algebraic power series in $\F_q ((z^{-1}))$ 
has unbounded partial quotients, in the sense that their degrees are unbounded. 
%, thus its sequence of partial quotients canot be automatic.
%(for the case $p=2$, take the example...). 
%However, the Diophantine nature of power series whose sequence of 
%partial quotients is automatic remain mysterious. Such a power series has necessarily 
%bounded partial quotients, but this observation does not help us much, as 
In the opposite direction, there exist
power series which are algebraic of degree at least $3$ and have bounded 
partial quotients (recall that analogous statements have not yet been proved, nor disproved, for 
real numbers): in 1976, Baum and Sweet \cite{BaSw76} 
proved that the continued fraction expansion of the unique solution $\xi_{BS}$ in $\F_2((z^{-1}))$
of the equation
$$
z X^3 + X + z= 0
$$
has all its partial quotients of degree at most $2$.

At the end of the 80s, Mend\`es France asked whether the sequence of 
partial quotients of an algebraic power series in $\F_q((z^{-1}))$ is $q$-automatic, as soon as 
it takes only finitely many different values. A positive answer has been 
given by Allouche \cite{Al88} and Allouche, B\'etr\'ema, and Shallit \cite{AlBeSh89} for 
the algebraic power series of degree at least $3$ in $\F_p ((z^{-1}))$ (here, $p \ge 3$) 
which have been 
constructed by Mills and Robbins \cite{MiRo86} and 
whose partial quotients are linear; see also Lasjaunias and Yao \cite{LasYao15,LasYao16,LasYao17}.      %%y 
However, in 1995 Mkaouar \cite{Mka95} (see also Yao \cite{Yao97} for an alternative proof) 
gave a negative answer to the question of Mend\`es France by
establishing that the sequence 
of partial quotients of the Baum--Sweet power series $\xi_{BS}$ is morphic, but not automatic. 
Recall that a sequence is called morphic if it is the image under a coding of a fixed point of a 
substitution. If the substitution can be chosen of constant length $k$, then the sequence 
is $k$-automatic. The equivalence between the two definitions of automatic sequences 
given here was established by Cobham \cite{Cob72}; see \cite{AlSh03}.

%Mills and Robbins gave examples, for each odd prime $p$, 
%of algebraic power series of degree at least $3$ in $\F_p ((z^{-1}))$ whose partial quotients are linear. 
%Subsequently, Allouche proved that the sequences of partial quotients of these examples 
%are automatic. See also... 

All this shows that the sequence of partial quotients of an algebraic power series may or may not 
be automatic. Conversely, very little is known about the Diophantine nature of a power series whose sequence of   %%y
partial quotients is automatic, but not ultimately periodic. 
The only contribution to this question is a recent work of Hu and Han \cite{HuHan}. 
They proved that, for any distinct non-constant polynomials $a$ and $b$ in $\F_2[z]$ whose 
sum of degrees is at most $7$, the power series whose sequence of partial quotients is the 
Thue--Morse sequence written over $\{a, b\}$ is algebraic of degree~$4$. 
Recall that the Thue--Morse word
$$
{\mathbf t} = t_0 t_1 t_2 \ldots = abbabaabbaababbabaababbaabbabaab \ldots 
%= 1  -1 -1  1  -1 1  1 -1  -1 1 1 -1 1 -1 -1 1  \ldots
$$
over $\{a, b\}$ is defined by $t_0 = a$,
$t_{2k} = t_k$ and $t_{2k + 1} = a$ (resp., $b$) if $t_k = b$ (resp., $a$),  for $k \ge 0$. 
The Thue--Morse sequence is the most famous automatic sequence and is 
not ultimately periodic. 
%We point out that the situation is completely different for real numbers. 

To establish their result, 
Hu and Han made use of a Guess 'n' Prove method and implemented 
a program which takes a pair $(a, b)$ of distinct non-constant polynomials as input and outputs its 
minimal defining polynomial and a complete proof. The time needed grows with the degrees 
of $a$ and $b$.  Hu and Han conjectured that their result holds 
more generally for every pair $(a, b)$ of distinct non-constant polynomials.

In the present paper, built on \cite{HuHan}, we confirm this conjecture. 
Our main results are stated in Section \ref{Res} and proved in Sections \ref{proofmain} 
and \ref{proofautres}. We consider 
higher degree exponents of approximation in Section \ref{higher}. 
Additional remarks are gathered in Section \ref{conc}. %%y 

\section{Results}   \label{Res} 

Let $a$ and $b$ be distinct non-constant polynomials in $\F_q [z]$ and let
$\xi_{q, a, b}$ denote the power series  in $\F_q((z^{-1}))$ 
whose sequence of partial quotients is the Thue--Morse sequence ${\mathbf t}$ over the alphabet $\{a, b\}$. 
Since ${\mathbf t}$ is not ultimately periodic, $\xi_{q, a, b}$ is transcendental or 
algebraic of degree at least $3$ over $\F_q (z)$. Furthermore, as ${\mathbf t}$ begins with infinitely 
palindromes, an argument given in the proof of Theorem \ref{thexp} 
shows that $\xi_{q, a, b}$ and its square are very well simultaneously approximable by 
rational fractions with the same denominator and that, consequently, $\xi_{q, a, b}$ cannot be 
algebraic of degree $3$. This proves that $\xi_{q, a, b}$ is transcendental or 
algebraic of degree at least $4$ over $\F_q(z)$.

Our main result asserts that, in the case $q=2$, any Thue--Morse continued fraction $\xi_{2,a,b} = \xi_{a,b}$
is algebraic of degree $4$ over $\F_2 (z)$. 

\begin{theorem} \label{mainth}
Let $a, b$ be non-constant, distinct polynomials in $\F_2 [z]$ and let 
$$
\xi_{a, b} = [0; a, b, b, a, b, a, a, b, b, a, a, b, a, b, b, a, \ldots ]
$$
be the power series  in $\F_2((z^{-1}))$ whose sequence of partial quotients is the Thue--Morse sequence 
${\mathbf t}$ over the alphabet $\{a, b\}$. Then, $\xi_{a, b}$ is algebraic of degree $4$ over $\F_2 (z)$.
More precisely, setting 
\begin{align*}
A_0&=b(a+b)(a^2 b^2 +a^2+b^2) + a^2 b^4, \\
A_1&=ab(a+b) (a^2 b^2 + a^2 + b^2), \\
A_2&=a^2 b^2 (a^2 b^2+a^2+b^2), \\
A_3&=A_1, \\
A_4&=a(a+b)(a^2 b^2 +a^2+b^2)+ a^2 b^4,
\end{align*}
%with
%$$
%g=1+a^2+b^2,
%$$
we have 
$$
A_4 \xi_{a, b}^4 + A_3 \xi_{a, b}^3 + A_2 \xi_{a, b}^2 + A_1 \xi_{a, b} + A_0 = 0.
$$
\end{theorem}

\noindent {\it Remark. } 
In Theorem \ref{mainth} the ground field $\F_2$ can be replacd by any field $K$ of characteristic $2$, and 
$a$ and $b$ by any non-constant, distinct polynomials in $K[z]$. 

\medskip

We define a norm $| \cdot |$
on the field $\F_q ((z^{-1}))$ by setting $|0| = 0$ and,
for any non-zero power series 
$\xi = \xi(z) = \sum_{h=-m}^{+ \infty} \, a_h z^{-h}$ with $a_{-m} \not= 0$, 
by setting $|\xi| = q^m$. We write $|| \xi ||$ for the norm
of the fractional part of $\xi$, that is, of the part of the series
which comprises only the negative powers of $z$.

For a power series $\xi$ in $\F_q ((z^{-1}))$, let
$$
C(\xi)  = \liminf_{Q \in \F_q [z] \setminus \{0\}}  \,
|Q| \cdot \Vert Q \xi \Vert 
$$
denote its Lagrange constant. Clearly, $C(\xi)$ is an element of
$$
\cL_q = \{q^{-k} :  k \ge 1 \} \cup \{ 0 \},
$$
and it is positive if and only if the 
degrees of the partial quotients of $\xi$ are bounded. 

If $[a_0; a_1, a_2,  \ldots]$ denotes the continued fraction expansion of $\xi$ 
and $p_n / q_n = [a_0; a_1, a_2, \ldots, a_n]$ for $n \ge 1$, then 
$$
|q_n \xi - p_n| = |q_{n+1}|^{-1} = |a_{n+1} q_n|^{-1}.
$$
Consequently, we have
$$
C(\xi) = 2^{- \limsup_{n \to + \infty}  \deg a_n}. 
$$
Exactly as in \cite{Mah49}, we can check that the root $z^{-1} + z^{-4} + z^{-4^2} + \ldots$ in 
$\F_2 ((z^{-1}))$ of $z X^4 + z X + 1$ is algebraic of degree $4$ and has unbounded 
partial quotients, thus its Lagrange constant is $0$.    %%y 
Since the degrees of $a$ and $b$ can be arbitrarily chosen in Theorem \ref{mainth}, we 
derive at once the following corollary. 

\begin{corollary}
There exist algebraic power series of degree $4$ in $\F_2((z^{-1}))$ with an arbitrarily prescribed 
Lagrange constant in $\cL_2$. 
\end{corollary}

More generally, we can define the spectrum of  badly approximable 
power series $\xi$ in $\F_q ((z^{-1}))$ as the set of positive integers $k$ such that $\xi$ 
has infinitely many partial quotients of degree $k$. 

\begin{problem}  \label{pbsp} 
Let $\cN$ be a finite set of positive integers. 
Does there exist a badly approximable 
power series $\xi$ in $\F_q ((z^{-1}))$ which is algebraic of degree at least $3$ 
and whose spectrum is equal to $\cN$? 
\end{problem}

Theorem \ref{mainth} answers positively Problem \ref{pbsp} for $q=2$ and 
any set $\cN$ of cardinality $2$.

For a power $q$ of a prime number $p$
the class $H(q)$ of hyperquadratic power series in $\F_q ((z^{-1}))$ 
has attracted special attention; see e.g.  \cite{Las00,Schm00}. 
It is composed of the irrational power series $\xi$ in $\F_q ((z^{-1}))$ 
for which there exist polynomials $A, B, C, D$ in $\F_q [z]$ such that 
$A D - B C \not= 0$ and
$$
\xi = {A \xi^{p^s} + B \over C \xi^{p^s} + D}.
$$
Once a power series is known to be algebraic, a natural question is to determine whether it belongs to the 
restricted class of hyperquadratic power series. 

In another direction, since the pioneering works of Kolchin \cite{Kol59} and 
Osgood \cite{Os75}, 
formal derivation has been used to study rational approximation to power series. 
By differentiating with respect to $z$ the minimal defining polynomial satisfied by 
an algebraic power series $\xi$ in $\F_q ((z^{-1}))$ of degree $d$, we
see that $\xi'$, where the $'$ indicates the derivation on $\F_q((z^{-1}))$ with respect to $z$
extending the derivation on $\F_q[z]$, can be expressed 
as a polynomial in $\xi$ of degree at most $d-1$. 
%Actually, an easy computation shows that 
%the latter polynomial has degree only $2$. 

We say that a power series $\xi$ in $\F_q ((z^{-1}))$ satisfies a Riccati differential equation if there are 
$A, B$ and $C$ in $\F_q (z)$ such that 
$$
\xi' = A \xi^2 + B \xi + C.
$$
Clearly, any cubic power series satisfies a Riccati equation. 
This is also the case for any hyperquadratic power series, but the 
converse does not hold; see e.g. \cite[Section 4]{LasBdM96}.

\begin{proposition}  \label{hypric} 
Let $a, b$ be non-constant, distinct polynomials in $\F_2 [z]$. 
Then, the Thue--Morse continued fraction 
$\xi_{a, b}$ in $\F_2((z^{-1}))$ is differential-quadratic: it satisfies the Riccati equation
$$
[(ab (a+b)) x]'= (ab)' (1+x^2).
$$
Furthermore, $\xi_{a, b}$ is not hyperquadratic. 
\end{proposition}

\noindent {\it Remark. } 
The Riccati equation in Proposition \ref{hypric} is equivalent to
$[(ab (a+b)) x]'= [ab (1+x^2)]'$. Its solutions are the power series $\xi$ 
such that $ab(a+b) \xi + ab(1+ \xi^2)$  is a square, that is, such that
$1+(a+b) \xi + \xi^2$ is a linear combinaison of terms $a^i b^j$, with $i$ and $j$ odd. 

\medskip

The strategy of the proof of Theorem \ref{mainth} is the following. 
After a careful study of the minimal defining polynomials $P_{a,b} (X)$ of $\xi_{a, b}$ found 
by Hu and Han \cite{HuHan} in the case where the sum of the degrees of $a$ and $b$ is at 
most $7$, we have guessed the coefficients, expressed in terms of $a$ and $b$, 
of the quartic polynomial $P_{a, b} (X)$ which vanishes at $\xi_{a, b}$, for {\it any} distinct, non-constant 
polynomials $a$ and $b$. It then only remained for us to 
check that $P_{a, b} (\xi_{a, b}) = 0$. 
This step is, however, much more difficult than it may seem to be. 
Denoting by $p_\ell / q_\ell$ (we drop the letters $a, b$) the $\ell$-th convergent of $\xi_{a,b}$ for $\ell \ge 1$, 
it is sufficient to check that $P_{a, b} (p_\ell / q_\ell)$ tends to $0$ as $\ell$ tends to infinity along 
a subsequence of the integers. 
We focus on the indices $\ell$ which are powers of $4$. 
%of the form $2^\ell$ and $2^\ell - 1$. 
We heavily use the properties of symmetry of the Thue--Morse sequence and we proceed by induction 
to show that $|P_{a, b} (p_{4^k} / q_{4^k})|$ is, for $k \ge 1$, 
less than $|q_{4^k}|^{-2}$ times some constant independent of~$k$.

\section{Higher degree exponents of approximation}   \label{higher}

Beside the rational approximation to a power series $\xi$ in $\F_q((z^{-1}))$,  
we often consider the simultaneous rational approximation of $\xi, \xi^2, \ldots , \xi^n$ 
by rational fractions with the same denominator, as well as small values of the 
linear form $b_0 + b_1 \xi + \ldots + b_n \xi^n$ with coefficients in $\F_q[z]$. This leads us to introduce
the exponents of approximation $w_n$ and $\lambda_n$, defined below. 
For a survey of recent results on these exponents evaluated at real numbers, the reader is directed to \cite{BuDurham}.

The height $H(P)$ of a polynomial $P(X) = b_n (z) X^n + \ldots + b_1(z) X + b_0 (z)$ 
over $\F_q[z]$ is the maximum of the absolute values of its 
coefficients, that is, of $|b_0|, |b_1|, \ldots , |b_n|$. 
%A power series in $C_\infty$ is called algebraic if it is a root of a nonzero 
%polynomial with coefficients in $\F_q[T]$. 
%Its height is then the height of its minimal 
%defining polynomial over $\F_q[T]$. 
Recall that the `fractional part' $\Vert \cdot \Vert$ is defined by
$$
\Bigl\Vert \sum_{h = -m}^{+ \infty} \, a_h z^{-h} \Bigr\Vert = \Bigl| \sum_{h=1}^{+ \infty} \, a_h z^{-h} \Bigr|,
$$
for every power series $\xi = \sum_{h = -m}^{+ \infty} \, a_h z^{-h}$ in $\F_q ((z^{-1}))$.

\begin{definition}
\label{Def:1.1}
Let $\xi$ be in $\F_q ((z^{-1}))$. Let $n \ge 1$ be an integer.
We denote by
$w_n (\xi)$ the supremum of the real numbers $w$ for which
$$
0 < |P(\xi )| < H(P)^{-w}
$$
has infinitely many solutions in polynomials $P(X)$ over $\F_q[z]$ of
degree at most $n$. 
We denote by
$\lambda_n (\xi)$ the supremum of the real numbers $\lambda$ for which
$$
0 < \max\{ \Vert Q(z) \xi \Vert, \ldots , \Vert Q(z) \xi^n \Vert \} < q^{-\lambda  \deg(Q)}
$$
has infinitely many solutions in polynomials $Q(z)$ in $\F_q[z]$. 
%We denote by
%$w_n^* (\xi)$ the supremum of the real numbers $w^*$ for which
%$$
%0 < |\xi - \alpha| < H(\alpha)^{-w^*-1}
%$$
%has infinitely many solutions in algebraic power series $\alpha$ in $\F_q ((T^{-1}))$ of
%degree at most $n$.
For positive real numbers $w$, $\lambda$, set
%For a power series $\xi$ and a positive real number $\mu$, put 
$$
B_n (\xi, w) = \liminf_{H(P) \to + \infty} \, H(P)^{w} \cdot |P (\xi) |
$$
and
$$
B'_n (\xi, \lambda) = \liminf_{|Q| \to + \infty} 
\, |Q|^{\lambda} \cdot \max\{ \Vert Q(z) \xi \Vert, \ldots , \Vert Q(z) \xi^n \Vert \}. 
$$
\end{definition}

Let $\xi$ be an algebraic power series in $\F_q ((z^{-1}))$ of degree $d \ge 2$. 
Let $n \ge 1$ be an integer. We briefly show how a Liouville-type argument 
allows us to bound $w_n (\xi)$ from above. 
Denote by $\xi_1 = \xi, \xi_2, \ldots , \xi_d$ its Galois conjugates. 
Let $P(X)$ be a non-zero polynomial in $\F_q [z] (X)$. 
Then, the product $P(\xi_1) \ldots P(\xi_d)$ is a nonzero element 
of $\F_q(z)$, whose absolute value is bounded from below by $c_1(\xi)$, which (as $c_2(\xi)$ 
and $c_3(\xi)$ below) is positive and depends only on $\xi$. 
Since $|P(\xi_j)|$, $j = 2, \ldots , d$, are bounded from above by $c_2 (\xi)$ times $H(P)$, 
we get that 
$$
|P(\xi)| > c_3 (\xi) H(P)^{-d+1}
$$
and we derive that 
\begin{equation} \label{liouv}
w_n (\xi) \le d-1, \quad n \ge 1. 
\end{equation}
In the particular case of $\xi_{a,b}$, this gives 
$w_3 (\xi_{a,b}) \le 3$, by Theorem \ref{mainth}. Actually, this inequality is an equality.

\begin{theorem}   \label{thexp}
The Thue--Morse power series $\xi_{a, b}$ in $\F_2((z^{-1}))$ satisfies
$$
\lambda_2 (\xi_{a, b}) = 1, \quad w_2 (\xi_{a, b}) = 3, \quad w_n (\xi_{a, b}) = 3, 
\quad \lambda_n (\xi_{a, b}) = \frac{1}{3},  \quad n \ge 3.
$$
Moreover, there are positive constants $c_4, c_5$, depending only on $a$ and $b$, such that
$$
|P(\xi_{a,b})| > c_4 H(P)^{-3}, \quad
\hbox{for every nonzero $P(X)$ in $\F_2[z] (X)$ of degree $\le 3$}, 
$$
and there are polynomials $P(X)$ in $\F_2[z] (X)$ of degree $2$ of arbitrarily large height such that 
$$
|P(\xi_{a,b})| < c_5 H(P)^{-3}. 
$$
\end{theorem}

\begin{proof}
Since the continued fraction expansion of $\xi_{a, b}$ begins with 
arbitrarily large palindromes, $\xi_{a, b}$ and its square are simultaneously 
well approximable by rational fractions with the same denominator. 
This argument appeared in \cite{AdBu07b} and we recall it below for the sake of completeness. 
For $k \ge 0$, let $p_k / q_k$ denote the $k$-th convergent to $\xi_{a, b}$. 
%Throughout this proof, we denote by $\bft = t_0 t_1 t_2 \ldots $ the 
Recall that the Thue--Morse word $\bft = t_0 t_1 t_2 \ldots $ over $\{a, b\}$
is the fixed point starting with $a$ 
of the uniform morphism $\tau$ defined by $\tau(a) = ab$ and $\tau(b)=ba$. 
Since the words $\tau^2 (a) = abba$ 
and $\tau^2 (b) = baab$, and the prefix of length $4$ of $\bft = a b b a \ldots$ are palindromes, 
every prefix of $\bft$  of length a power of $4$ 
is a palindrome. Put 
$$
\cM_a = \begin{pmatrix} a & 1 \\ 1 & 0\\ \end{pmatrix}, 
\quad
\cM_b = \begin{pmatrix} b & 1 \\ 1 & 0\\ \end{pmatrix}.
$$
Then, for $k \ge 1$, we have
$$
\cM_{t_0} \cM_{t_1} \ldots \cM_{t_{k-1}} = 
\begin{pmatrix} q_k & q_{k-1} \\ p_k & p_{k-1} \\ \end{pmatrix}.
$$
Take $k = 4^\ell$ for some positive integer $\ell$. 
The matrix $\cM_{t_0} \cM_{t_1} \ldots \cM_{t_{k-1}}$ is symmetrical 
since $t_0 t_1 \ldots t_{k-1}$ is a palindrome. Consequently, $p_k = q_{k-1}$. 
It then follows from $t_{k-1} = a$, $t_k = b$ and the theory of continued fraction that 
$$
\Bigl| \xi_{a, b} - \frac{p_k}{q_k} \Bigr| = \frac{1}{|b| |q_k|^2}, 
\quad
\Bigl| \xi_{a, b} - \frac{p_{k-1}}{q_{k-1}} \Bigr| = \frac{1}{|a| |q_{k-1}|^2} = \frac{1}{|q_{k-1}| \cdot |q_k|},
$$
thus
$$
\Bigl| \xi^2_{a, b} - \frac{p_{k-1}}{q_k} \Bigr| = 
\Bigl| \Bigl( \xi_{a, b} - \frac{p_{k}}{q_{k}} \Bigr)  \Bigl( \xi_{a, b} + \frac{p_{k-1}}{q_{k-1}} \Bigr) + 
\frac{\xi_{a,b}}{q_{k-1} q_k} \Bigr| = \frac{1}{|q_k|^2}, 
%\frac{1}{|q_{k-1}| \cdot |q_k|}, 
$$
since $|\xi_{a,b}| = 1 / |a|$. 
%Since $|q_k| = |a| \cdot |q_{k-1}|$, we get that 
We conclude that 
$$
\max\{ \|q_k \xi_{a,b} \|,  \|q_k \xi_{a,b}^2 \| \} = \frac{1}{|q_k|}. 
$$
This gives $\lambda_2 (\xi_{a, b}) \ge 1$ and there is equality since 
$\lambda_1 (\xi_{a, b}) = 1$. Furthermore, the quantity 
$B'_2 (\xi_{a, b}, 1)$ is finite, since $\xi_{a, b}$ has bounded partial quotients.   %%y

By the power series field analogue 
of a classical transference inequality established by Aggarwal \cite{Agg68}, 
we immediately obtain that $B_2 (\xi_{a, b}, 3)$ is finite, thus
$w_2 (\xi_{a, b}) \ge 3$. 

This lower bound holds as well for $w_2 (\xi_{q, a, b})$. Combined with \eqref{liouv}, this 
shows that $\xi_{q,a,b}$ is transcendental or algebraic of degree at least $4$, thereby 
establishing the assertion above Theorem \ref{mainth}.

Since $\xi_{a, b}$ is algebraic of degree $4$, by Theorem \ref{mainth}, 
the Liouville-type result obtained below Definition \ref{Def:1.1} combined with %%y 
the lower bound $w_2 (\xi_{a, b}) \ge 3$ implies that $w_2 (\xi_{a, b}) = w_3 (\xi_{a, b}) = 3$ and %%y 
$$
w_n (\xi_{a, b}) = 3, \quad \lambda_n (\xi_{a, b}) = \frac{1}{3}, \quad n \ge 3,
$$
the value of $\lambda_3 (\xi_{a, b})$ being a consequence of a transference inequality established in \cite{Agg68}. 
%Since $w_3 (\xi_{a, b}) \ge w_2 (\xi_{a,b})$, we deduce that $w_2 (\xi_{a,b}) = 3$. 
\end{proof}

\section{Proof of Theorem \ref{mainth}}   \label{proofmain}

We keep the notation of Theorem \ref{mainth} and  %%y 
check that $\xi_{a,b}$ is a root of the equation
$$
A_4 X^4 + \ldots + A_1 X + A_0 = 0. 
$$
In view of the discussion at the beginning of Section \ref{Res}, 
this implies that $\xi_{a,b}$ is algebraic of degree $4$. 
%However, as the Thue-Morse sequence is not ultimately periodic, $\xi_{a,b}$ is not 
%quadratic. Furthermore, as explained in the proof of Theorem \ref{thexp}, the continued 
%fraction of $\xi_{a,b}$ begins with arbitrarily large palindromes, thus we have 
%$w_2 (\xi_{a, b}) \ge 3$ and $\xi_{a,b}$ cannot be algebraic of degree $3$. 
%We conclude that $\xi_{a,b}$ is algebraic of degree exactly $4$. 

The computation is easier if we replace $a$ and $b$ by their inverses, that is, 
if we consider the continued fraction
\begin{equation} \label{zeta}
\zeta = [0; a^{-1}, b^{-1}, b^{-1}, a^{-1}, b^{-1}, a^{-1}, a^{-1}, b^{-1}, \ldots]. 
\end{equation}

\begin{definition}
Set
$$
M_0(a,b)= a \begin{pmatrix} a^{-1} & 1 \\ 1 & 0\\ \end{pmatrix} 
= \begin{pmatrix} 1 & a \\ a & 0\\ \end{pmatrix}
$$
and
$$
M_k(a,b)=M_{k-1}(a,b) \cdot  M_{k-1}(b,a)^2 \cdot  M_{k-1}(a,b), \quad k \ge 1,
$$
$$
N_k (a, b) = M_{k-1}(a,b) \cdot  M_{k-1}(b,a), \quad k\ge 1. 
$$
\end{definition}

Observe that 
$$
M_1(a,b)
= \begin{pmatrix} a^2 b^2 + b^2 + 1 & a^2 b + a b^2 + a \\ a^2 b + a b^2 + a & a^2 b^2 + a^2 \\ \end{pmatrix}
$$
and that, for $k \ge 1$, the matrix $M_k(a, b)$ is symmetric and of the form
$$
%M_{k} (a, b) =
\begin{pmatrix}
	1 + \cdots + (ab)^{2^{2k-1}} & a + \cdots + (a+b)^{(2^{2k-1}+1)/3} (ab)^{2^{2k-2}}\\
	 a + \cdots + (a+b)^{(2^{2k-1}+1)/3} (ab)^{2^{2k-2}} & a^2 + \cdots + (ab)^{2^{2k-1}}
\end{pmatrix}. 
$$
An immediate induction shows that its entries are polynomials in $a$ and $b$, whose degrees
in $a$ (resp., in $b$) are at most equal to $2^{2k-1}$.  

For a polynomial $M (a, b)$ (or a matrix $M (a, b)$ whose coefficients are polynomials) 
in the variables $a$ and $b$, write $\tiM (a, b)$ the image of $M(a, b)$ under the involution 
which exchanges $a$ and $b$, that is, $\tiM(a,b) = M(b, a)$. 

It follows from the definition of $N_k (a, b)$ that its upper left and lower right entries 
are symmetrical in $a$ and $b$, while its upper right entry is obtained from its lower left one by exchanging 
$a$ and $b$. We introduce some further notation. 

\begin{definition}
Write
$$
M_k(a,b)=\begin{pmatrix} Q_k(a,b) & P_k(a,b) \\ P_k(a,b) & R_k(a,b) \end{pmatrix}, \quad k \ge 0,
$$
and
$$
N_k(a,b)=\begin{pmatrix} U_k(a,b) & V_k(a,b) \\ V_k(b,a) & W_k(a,b) \end{pmatrix}, \quad k \ge 1,
$$
To shorten the notation, we simply write 
$$
M_k = \begin{pmatrix} Q_k & P_k  \\ P_k  & R_k  \end{pmatrix}, \quad 
N_k =\begin{pmatrix} U_k  & V_k \\ \tiV_k  & W_k  \end{pmatrix}, \quad k \ge 1,
$$
and observe  that 
$$
U_k = \tiU_k, \quad W_k = \tiW_k, \quad k \ge 1. 
$$
\end{definition}

Set
$$
\tau = 1 + a + b
$$
and define
$$
n_k = 2^{2k - 1}, \quad k \ge 1.  
$$

First, we establish several relations between the entries of $M_k$ and $\tiM_k$. 

\begin{proposition}   \label{PQR} 
For $k \ge 1$, we have
\begin{align*}
	Q_k + R_k &= \tau^{(2 n_k  + 2) / 3}, \\
	P_k + \tiP_k &= (1 + \tau)   \tau^{(2 n_k - 4) / 3 },\\
	Q_k + \tiR_k &= \tau^{(2 n_k - 4) / 3 }.
\end{align*}
\end{proposition}

\begin{proof}
We proceed by induction. Check that
$$
Q_1 + R_1 = a^2 + b^2 + 1 = \tau^2,
$$
$$
P_1 + \tiP_1 = a + b = 1 + \tau = (1 + \tau) \tau^0,
$$
and
$$
Q_1 + \tiR_1 = 2 a^2 b^2 + 2 b^2 + 1 = 1 = \tau^0.
$$
Set
$$
m_k = (2 n_k + 2) / 3 =  (2^{2k} + 2) / 3, \quad k \ge 0,
$$
and observe that 
$$
m_k = 4 m_{k-1} - 2, \quad k \ge 1. 
$$
Since $M_k = N_k {\widetilde N_k}$, we have
$$
Q_k = U_k^2 + V_k^2, \quad R_k = \tiV_k^2 + W_k^2, \quad
P_k = U_k \tiV_k + V_k W_k, \quad k \ge 1. 
$$
Let $k \ge 2$ be an integer and assume that the proposition holds for the index $k-1$. 
Since $\tau$ is invariant by exchanging $a$ and $b$, 
it follows from the induction hypothesis that 
$$
R_{k-1} + Q_{k-1} = \tiR_{k-1}  + \tiQ_{k-1} = \tau^{m_{k-1}}.
$$
Consequently, we get
\begin{align*}
V_k + \tiV_k & = Q_{k-1} \tiP_{k-1} + P_{k-1} \tiR_{k-1} + \tiQ_{k-1} P_{k-1} + \tiP_{k-1} R_{k-1}  \\
& = \tiP_{k-1} (Q_{k-1} + R_{k-1}) + P_{k-1} (\tiQ_{k-1} + \tiR_{k-1}) \\
& = ( P_{k-1} +  \tiP_{k-1} ) (Q_{k-1} + R_{k-1}) \\
& = (1 + \tau)  \tau^{m_{k-1} - 2}  \tau^{m_{k-1}}.
\end{align*}
Likewise, we have
\begin{align*}
U_k + W_k & = Q_{k-1} \tiQ_{k-1}  + R_{k-1} \tiR_{k-1} \\
& = Q_{k-1} (\tau^{m_{k-1}} + \tiR_{k-1}) + \tiR_{k-1} (Q_{k-1} + \tau^{m_{k-1}}) \\ 
& = \tau^{m_{k-1}} (Q_{k-1} + \tiR_{k-1}) = \tau^{2 m_{k-1} - 2}. 
\end{align*}
This gives
$$
P_k + \tiP_k  = (V_k + \tiV_k ) (U_k + W_k) =  (1 + \tau) \tau^{4 m_{k-1} - 4}
= (1 + \tau) \tau^{m_{k} - 2},
$$
as expected. 

We get
\begin{align*}
Q_k + R_k & = U_k^2 + V_k^2 + \tiV_k^2 + W_k^2  \\
& = (U_k + W_k)^2 + (V_k + \tiV_k)^2 \\ 
& = \tau^{4 m_{k-1} - 4} + (1 + \tau^2) \tau^{4 m_{k-1} - 4} = \tau^{4 m_{k-1} - 2} = \tau^{ m_{k}}. 
\end{align*}
Finally, we check that
\begin{align*}
Q_k + \tiR_k & = U_k^2 + W_k^2  \\
& = (U_k + W_k)^2  = \tau^{4 m_{k-1} - 4}  = \tau^{ m_{k} -2},
\end{align*}
and the proof is complete. 
\end{proof}

In the next lemma, we express $P_k, Q_k, R_k$ in terms of a new auxiliary quantity 
denoted by $Z_k$.

\begin{lemma}    \label{ZPQR}
Let $k$ be a positive integer and define
$$
Z_k = U_k + V_k.
$$
Then, we have
$$
Q_k = Z_k^2, \quad R_k = Z_k^2 + \tau^{ (2 n_k + 2) / 3 },
\quad
P_k = Z_k^2 + \tau^{(n_k + 1) / 3} Z_k + (ab)^{n_k},
$$
and
$$
\tiZ_k = Z_k + (1 + \tau) \tau^{(n_k - 2)/ 3}.
$$
\end{lemma}

\begin{proof}
Since $Q_k = U_k^2 + V_k^2$, we get immediately that $Q_k = Z_k^2$. 
The expression of $R_k$ in terms of $Z_k$ follows then from Proposition \ref{PQR}.

The matrix $M_k$ is the product, in a suitable order, of $n_k$ copies of $M_0(a,b)$ 
and $n_k$ copies of $M_0(b,a)$, whose determinants are respectively equal to $a^2$ and $b^2$. 
The computation of the determinant of $M_k$ then gives the relation
$$
Q_k R_k + P_k^2 =  (ab)^{2 n_k}. 
$$
Consequently,
$$
P_k^2 = Q_k (Q_k + \tau^{(2 n_k + 2) / 3}) + (ab)^{2 n_k},
$$
thus
$$
P_k = Z_k^2 + \tau^{(n_k + 1) / 3} Z_k + (ab)^{n_k}. 
$$

It follows from Proposition \ref{PQR} that 
$$
\tiQ_k + Q_k = \tau^{(2 n_k  + 2) / 3} + \tau^{(2 n_k  - 4) / 3} = (1 + \tau^2) \tau^{(2 n_k  - 4) / 3},  
$$
which implies the last equality of the lemma. 
\end{proof}

We are now in position to express $Z_{k+1}$ in terms of $k$ and $Z_k$.

\begin{proposition}  \label{reczk}
For $k \ge 1$, we have 
%We check by induction that 
%$$
%Q_{k+1} = \tau^{2n_k} Q_k + \tau^{(4n_k - 2)/3} (ab)^{2n_k} + (ab)^{4n_k},
%$$
%thus
$$
Z_{k+1} = \tau^{n_k} Z_k + \tau^{(2n_k-1)/3} (ab)^{n_k} + (ab)^{2n_k}. 
$$
\end{proposition}

\begin{proof}
We proceed by induction. 
Observe that 
$$
Z_1 = U_1 + V_1 = ab + b + 1 
$$
and
$$
Z_2 =  a^4 b^4 + a^3 b + a^2 b^2 + a^2 + ab + ab^3 + b^2 + 1 + a^3 b^2 + a^2 b + a^2 b^3 + b + b^3.
$$
Since
$$
\tau^2 Z_1 + \tau (ab)^{2} + (ab)^4 =
(1 + a^2 + b^2) ( ab + b + 1 ) + (1 + a + b) a^2 b^2 + a^4 b^4,
$$
we see that the proposition holds for $k=1$. 
Assuming that it holds for a positive integer $k$, we get
\begin{align*}
Z_{k+1} &=P_k \tiP_k + Q_k \tiQ_k  + Q_k \tiP_k + P_k \tiR_k\\
%&=P_k( \tiP_k +\tiR_k) + Q_k( \tiQ_k  +  \tiP_k)\\
%&=P_k( \tiP_k +\tiR_k +\tiQ_k +\tiQ_k ) + Q_k( \tiQ_k  +  \tiP_k)\\
&=P_k( \tiQ_k  +\tiR_k) + (P_k+Q_k)( \tiQ_k +\tiP_k)\\
&=\Bigl(Z_k^2 + \tau^{(n_k + 1)/3} Z_k + (ab)^{n_k}\Bigr)
 \tau^{(2n_k + 2)/3} \\
 & \quad +
 \Bigl(\tau^{(n_k + 1)/3} Z_k + (ab)^{n_k}\Bigr)
\Bigl( \tau^{(n_k + 1)/3} \tiZ_k  + (ab)^{n_k}\Bigr)\\
%&=(Z_k^2 + \tau^{(n_k + 1)/3} Z_k + (ab)^{n_k})
% \tau^{(2n_k + 2)/3} \\
%& \quad +
% (\tau^{(n_k + 1)/3} Z_k + (ab)^{n_k})
%( \tau^{(n_k + 1)/3}  Z_k+(1 + \tau) \tau^{(2 n_k - 1)/3}  + (ab)^{n_k})\\
&=\tau^{(2n_k + 2)/3} Z_k^2 + \tau^{n_k + 1} Z_k + (ab)^{n_k}\tau^{(2n_k + 2)/3} \\
& \quad +
 \tau^{(n_k + 1)/3} Z_k 
\Bigl( \tau^{(n_k + 1)/3}  Z_k+(1 + \tau) \tau^{(2 n_k - 1)/3} + (ab)^{n_k}\Bigr)\\
& \quad +
 (ab)^{n_k} 
\Bigl( \tau^{(n_k + 1)/3}  Z_k+(1 + \tau) \tau^{(2 n_k - 1)/3  } + (ab)^{n_k}\Bigr)\\
&=\tau^{n_k + 1}  Z_k + (ab)^{n_k}\tau^{(2 n_k + 2)/3} + \tau^{n_k} Z_k  ( 1 + \tau )\\
& \quad +
 (ab)^{n_k}
\Bigl((1 + \tau) \tau^{(2n_k - 1)/3}   + (ab)^{n_k}\Bigr)\\
%&=  \tau^{n_k} Z_k ( \tau + (1 + \tau)  )
%& \quad 
%+ \tau^{(2 n_k - 1)/3 } (ab)^{n_k}  (\tau + (1 + \tau) ) + (ab)^{2n_k}\\
&= \tau^{n_k} Z_k  + \tau^{(2 n_k - 1)/3 }  (ab)^{n_k} + (ab)^{2n_k}. 
\end{align*}
This completes the proof.
\end{proof}

We deduce from Proposition \ref{reczk} an equation of degree $4$ satisfied by $Z_k$.

\begin{proposition}   \label{polzk}
For $k \ge 1$, we have
\begin{align*}
Z_k^4 +  \tau^{(2n_k -1)/3  }  Z_k^2 +(a+b) & \tau^{n_k -1} Z_k   +(b(a+b)\tau^2 + a^2)\tau^{(4n_k -8)/3} \\
& + (a + b) (ab)^{n_k} \tau^{(2n_k - 4) / 3} + (ab)^{2 n_k} = 0.
\end{align*}
%There exists a polynomial $T_k (X, Y)$ with coefficients in $\F_2$ 
%such that $\delta_k = (ab)^{2 n_k} T_k(a, b)$. 
\end{proposition}

\begin{proof}
For $k \ge 1$, set 
$$
\Delta_k
=Z_k^4 +  \tau^{(2n_k -1)/3  }  Z_k^2 +(a+b)  \tau^{n_k -1} Z_k  +(b(a+b)\tau^2 + a^2)\tau^{(4n_k -8)/3}.
%\end{align*}
$$
Note that 
%{\bf [J'ai v\'erifi\'e !]}
\begin{align*}
\Delta_1 & = (1 + b^4 + a^4 b^4) + (1 + a + b) (1 + b^2 + a^2 b^2 + (a + b) (1 + b + ab )) \\
& \quad + b (a + b) (1 + a^2 + b^2) + a^2  \\
& = a^4 b^4 + a^2 b^3 + a^3 b^2 = (a + b) (ab)^{n_1} \tau^{(2n_1 - 4) / 3} + (ab)^{2 n_1}. 
\end{align*}
Thus, Proposition \ref{polzk} holds for $k = 1$. 
%Assume that it holds for a positive integer $k$. 
Let $k$ be a positive integer. 
Since $n_{k+1} = 4 n_k$, it follows from Proposition \ref{reczk} that %%y 
\begin{align*}
\Delta_{k+1}
&=Z_{k+1}^4 +  \tau^{(8n_k-1)/3  }  Z_{k+1}^2 +(a+b)  \tau^{4n_k-1} Z_{k+1}\\
& \quad +(b(a+b)\tau^2 + a^2)\tau^{(16n_k-8)/3}\\
&=\Bigl(\tau^{n_k} Z_k + \tau^{(2n_k-1)/3} (ab)^{n_k} + (ab)^{2n_k}\Bigr)^4\\
& \quad +  \tau^{(8n_k-1)/3  } \Bigl(\tau^{n_k} Z_k + \tau^{(2n_k-1)/3} (ab)^{n_k} + (ab)^{2n_k}\Bigr)^2\\
& \quad +(a+b)  \tau^{4n_k-1} \Bigl(\tau^{n_k} Z_k + \tau^{(2n_k-1)/3} (ab)^{n_k} + (ab)^{2n_k}\Bigr)\\
& \quad +(b(a+b)\tau^2 + a^2)\tau^{(16n_k-8)/3}\\
&=\tau^{4n_k} Z_k^4 + \tau^{(8n_k-4)/3} (ab)^{4n_k} +  (ab)^{8 n_k} \\
& \quad +  \tau^{(8n_k-1)/3  } \Bigl(\tau^{2n_k} Z_k^2 + \tau^{(4n_k-2)/3} (ab)^{2n_k} + (ab)^{4n_k}\Bigr)\\
& \quad +(a+b)  \tau^{4n_k-1} \Bigl(\tau^{n_k} Z_k + \tau^{(2n_k-1)/3} (ab)^{n_k} + (ab)^{2n_k}\Bigr)\\
& \quad +(b(a+b)\tau^2 + a^2)\tau^{(16n_k-8)/3}\\
&=\tau^{4n_k} Z_k^4 +  \tau^{(8n_k-1)/3  } (\tau^{2n_k} Z_k^2 ) +(a+b)  \tau^{4n_k-1} (\tau^{n_k} Z_k )\\
& \quad +(b(a+b)\tau^2 + a^2)\tau^{(16n_k-8)/3} +  \tau^{(8n_k-4)/3} (ab)^{4n_k}\\
& \quad +  \tau^{(8n_k-1)/3  } ( \tau^{(4n_k-2)/3} (ab)^{2n_k} + (ab)^{4n_k})\\
& \quad +(a+b)  \tau^{4n_k-1} (\tau^{(2n_k-1)/3} (ab)^{n_k} + (ab)^{2n_k}) +  (ab)^{8 n_k} \\
&=\tau^{4n_k} \Bigl( Z_k^4 +  \tau^{(2n_k-1)/3  }  Z_k^2 +(a+b)  \tau^{n_k-1}  Z_k\\
& \quad +(b(a+b)\tau^2 + a^2)\tau^{(4n_k-8)/3}   \Bigr) +  \tau^{(8n_k-4)/3}(1+\tau) (ab)^{4n_k}\\
& \quad +  \tau^{(8n_k-1)/3  }  \tau^{(4n_k-2)/3} (ab)^{2n_k}\\
& \quad +(a+b)  \tau^{4n_k-1} (\tau^{(2n_k-1)/3} (ab)^{n_k} + (ab)^{2n_k}) +  (ab)^{8 n_k} \\
&=\tau^{4n_k} \Delta_k +  \tau^{(8n_k-4)/3}(a+b) (ab)^{4n_k} +  \tau^{4n_k } (ab)^{2n_k}\\
& \quad +(a+b)  \tau^{4n_k-1} (\tau^{(2n_k-1)/3} (ab)^{n_k} ) +  (ab)^{8 n_k}.
\end{align*}
Assuming that
$$
\Delta_k = (a + b) (ab)^{n_k} \tau^{(2n_k - 4) / 3} + (ab)^{2 n_k}, 
$$
we deduce that 
\begin{align*}
\Delta_{k+1} &=\tau^{4n_k} (a + b) (ab)^{n_k} \tau^{(2n_k - 4) / 3} 
+ \tau^{4 n_k}   (ab)^{2 n_k} \\
& \quad +  \tau^{(8n_k-4)/3}(a+b) (ab)^{4n_k} +  \tau^{4n_k } (ab)^{2n_k}\\
& \quad +(a+b)  \tau^{4n_k-1} (\tau^{(2n_k-1)/3} (ab)^{n_k} ) +  (ab)^{8 n_k} \\
& =  \tau^{(8n_k-4)/3}(a+b) (ab)^{4n_k}  +  (ab)^{8 n_k}. 
\end{align*}
Since $n_{k+1} = 4 n_k$, this shows that Proposition \ref{polzk} holds for every positive integer~$k$. 
\end{proof}

Set
\begin{align*}
a_2&=(1+a+b)^2, \\
a_0&=a(a+b)a_2 + a^2, \\
a_1=a_3&=(a+b)a_2, \\
a_4&=b(a+b)a_2+ a^2. 
\end{align*}

To prove that 
$$
a_4 \zeta^4 + a_3 \zeta^3 + a_2 \zeta^2 + a_1 \zeta + a_0  = 0,
$$
with $\zeta$ as in \eqref{zeta}, 
it is sufficient to show that the absolute value of 
$$
\delta_k := a_4 P_k^4 + a_3 P_k^3 Q_k + a_2 P_k^2 Q_k^2 + a_1 P_k Q_k^3 + a_0 Q_k^4 
$$
tends to $0$ as $k$ tends to infinity.  
The next proposition is more precise.

\begin{proposition}
For $k \ge 1$, there exists a polynomial $T_k (X, Y)$ with coefficients in $\F_2$ 
such that $\delta_k = (ab)^{2 n_k} T_k(a, b)$. 
\end{proposition}

\begin{proof}
We express $\delta_k$ in terms of $Z_k$. The lengthy computation 
based on Lemma \ref{ZPQR} gives.   %%y 
\begin{align*}
\delta_k&=a_4 P_k^4 + a_3 P_k^3Q_k + a_2 P_k^2 Q_k^2 + a_1 P_k Q_k^3 + a_0 Q_k^4\\
&= a_4 (P_k^4+ Q_k^4) +(a+b)^2\tau^2 Q_k^4 + (a+b)\tau^2 P_kQ_k( P_k^2+Q_k^2 ) + \tau^2 P_k^2 Q_k^2\\
&= a_4 (\tau^{(n_k+1)/3} Z_k+(ab)^{n_k})^4 +(a+b)^2\tau^2 Q_k^4\\
& \quad + (a+b)\tau^2 P_kQ_k( \tau^{(n_k+1)/3} Z_k +(ab)^{n_k} )^2 + \tau^2 P_k^2 Q_k^2\\
&= a_4 (\tau^{(n_k+1)/3} Z_k+(ab)^{n_k})^4 +(a+b)^2\tau^2 Z_k^8\\
& \quad + (a+b)\tau^2 P_kZ_k^2( \tau^{(n_k+1)/3} Z_k +(ab)^{n_k} )^2 + \tau^2 P_k^2 Z_k^4\\
&= a_4 (\tau^{(n_k+1)/3} Z_k+(ab)^{n_k})^4 +(a+b)^2\tau^2 Z_k^8\\
& \quad + (a+b)\tau^2 (Z_k^2+\tau^{(n_k+1)/3} Z_k + (ab)^{n_k})Z_k^2( \tau^{(n_k+1)/3} Z_k +(ab)^{n_k} )^2\\
& \quad + \tau^2 (Z_k^2+\tau^{(n_k+1)/3} Z_k + (ab)^{n_k})^2 Z_k^4\\
%&=a_4 (\tau^{(n_k+1)/3} Z_k+(ab)^{n_k})^4 +(a+b)^2\tau^2 Z_k^8\\
%& \quad + (a+b)\tau^2 (Z_k^2+\tau^{(n_k+1)/3} Z_k + (ab)^{n_k})Z_k^2( \tau^{(n_k+1)/3} Z_k +(ab)^{n_k} )^2\\
%& \quad + \tau^2 (Z_k^2+\tau^{(n_k+1)/3} Z_k + (ab)^{n_k})^2 Z_k^4\\
%&=a_4 \tau^{(4 n_k+4)/3} Z_k^4+ a_4(ab)^{4{n_k}} +(a+b)^2\tau^2 Z_k^8\\
%& \quad + (a+b)\tau^2 (Z_k^2+\tau^{(n_k+1)/3} Z_k + (ab)^{n_k})Z_k^2( \tau^{(n_k+1)/3} Z_k +(ab)^{n_k} )^2\\
%& \quad + \tau^2 (Z_k^2+\tau^{(n_k+1)/3} Z_k + (ab)^{n_k})^2 Z_k^4\\
&=a_4 \tau^{(4 n_k+4)/3} Z_k^4+ a_4(ab)^{4{n_k}} +(a+b)^2\tau^2 Z_k^8\\
& \quad + (a+b)\tau^2 Z_k^2 (Z_k^2+\tau^{(n_k+1)/3} Z_k + (ab)^{n_k})( \tau^{(2 n_k+2)/3}  Z_k^2 +(ab)^{2{n_k}} )\\
& \quad + \tau^2 (Z_k^4+\tau^{(2 n_k+2)/3}  Z_k^2 + (ab)^{2{n_k}}) Z_k^4\\
&=a_4 \tau^{(4 n_k+4)/3} Z_k^4+ a_4(ab)^{4{n_k}} +(a+b)^2\tau^2 Z_k^8\\
& \quad + (a+b)\tau^2 Z_k^2 \bigl( \tau^{(2 n_k+2)/3} Z_k^4 +\tau^{3 n_k+1}  Z_k^3 + \tau^{(2 n_k+2)/3} Z_k^2(ab)^{n_k} \bigr) \\
& \quad +  (a+b)\tau^2 Z_k^2 \bigl( (ab)^{2{n_k}} Z_k^2 +(ab)^{2{n_k}}\tau^{(n_k+1)/3} Z_k + (ab)^{3{n_k}} \bigr)\\
& \quad + \tau^2 (Z_k^8+\tau^{(2 n_k+2)/3}  Z_k^6 + (ab)^{2{n_k}} Z_k^4 )\\
%&=a_4 \tau^{(4 n_k+4)/3} Z_k^4+ a_4(ab)^{4{n_k}} +(a+b)^2\tau^2 Z_k^8\\
%& \quad +  (a+b)\tau^2 Z_k^2 \tau^{(2 n_k+2)/3} Z_k^4 +(a+b)\tau^2 Z_k^2 \tau^{3 n_k+1}  Z_k^3 +(a+b)\tau^2 Z_k^2  \tau^{(2 n_k+2)/3} Z_k^2(ab)^{n_k}\\
%& \quad +   (a+b)\tau^2 Z_k^2 (ab)^{2{n_k}} Z_k^2 +(a+b)\tau^2 Z_k^2 (ab)^{2{n_k}}\tau^{(n_k+1)/3} Z_k +(a+b)\tau^2 Z_k^2  (ab)^{3{n_k}}\\
%& \quad + \tau^2 Z_k^8+\tau^2 \tau^{(2 n_k+2)/3}  Z_k^6 +\tau^2  (ab)^{2{n_k}} Z_k^4\\
&=a_4 \tau^{(4 n_k+4)/3} Z_k^4+ a_4(ab)^{4{n_k}} +(a+b)^2\tau^2 Z_k^8\\
& \quad +  (a+b)\tau^2 \tau^{(2 n_k+2)/3} Z_k^6 +(a+b)\tau^2  \tau^{3 n_k+1}  Z_k^5 +(a+b)\tau^2   \tau^{(2 n_k+2)/3} Z_k^4(ab)^{n_k}\\
& \quad +   (a+b)\tau^2  (ab)^{2{n_k}} Z_k^4 +(a+b)\tau^2  (ab)^{2{n_k}}\tau^{(n_k+1)/3} Z_k^3 +(a+b)\tau^2 Z_k^2  (ab)^{3{n_k}}\\
& \quad + \tau^2 Z_k^8+\tau^2 \tau^{(2 n_k+2)/3}  Z_k^6 +\tau^2  (ab)^{2{n_k}} Z_k^4 
\\
&= \bigl(a_4 \tau^{(4 n_k+4)/3} + \tau^2  (ab)^{2{n_k}}+ (a+b)\tau^2   \tau^{(2 n_k+2)/3} (ab)^{n_k} +   (a+b)\tau^2  (ab)^{2{n_k}} \bigr) Z_k^4 \\
& \quad + a_4(ab)^{4{n_k}}
%& \quad 
+[(a+b)^2\tau^2 +\tau^2 ] Z_k^8\\
& \quad + [ (a+b)\tau^2 \tau^{(2 n_k+2)/3}  +\tau^2 \tau^{(2 n_k+2)/3} ] Z_k^6 +(a+b)\tau^2  \tau^{3 n_k+1}  Z_k^5\\
& \quad +(a+b)\tau^2  (ab)^{2{n_k}}\tau^{(n_k+1)/3} Z_k^3 +(a+b)\tau^2 Z_k^2  (ab)^{3{n_k}} \\
&= \bigl( (b(a+b)\tau^2 + a^2)\tau^{(4 n_k+4)/3} +  (a+b)   \tau^{(2 n_k+8)/3} (ab)^{n_k} +   \tau^3  (ab)^{2{n_k}} \bigr) Z_k^4 \\
& \quad + (b(a+b)\tau^2 + a^2)(ab)^{4{n_k}} 
%& \quad 
+ \tau^4  Z_k^8 +  \tau^{(2 n_k+8)/3}   Z_k^6 +(a+b)  \tau^{3 n_k+3}  Z_k^5\\
& \quad +(a+b)   (ab)^{2{n_k}}\tau^{(n_k+7)/3} Z_k^3 +(a+b)\tau^2 Z_k^2  (ab)^{3{n_k}}\\
&= \tau^4  Z_k^8 +  \tau^{(2{n_k}+ 11)/3 }  Z_k^6 +(a+b)  \tau^{{n_k}+3} Z_k^5\\
& \quad +[(b(a+b)\tau^2 + a^2)\tau^{(4{n_k}-5)/3} +  (a+b)   \tau^{(2{n_k}-1)/3}(ab)^{n_k} +    (ab)^{2{n_k}} ]\tau^3 Z_k^4\\
& \quad +(a+b)  (ab)^{2{n_k}}\tau^{({n_k}+7)/3} Z_k^3 +(a+b)\tau^2  (ab)^{3{n_k}}Z_k^2\\
& \quad + (b(a+b)\tau^2 + a^2)(ab)^{4{n_k}}.
\end{align*}
This can be rewritten as
%A tedious computation shows that 
\begin{align*}
\delta_k & = 
%\tau^4 Z_k^8 +  \tau^{(2n_k+2)/3 +3 }  Z_k^6 +(a+b)  \tau^{n_k+3} Z_k^5\\
%& \quad + 
%\bigl((b(a+b)\tau^2 + a^2)\tau^{(4n_k-5)/3} +  (a+b)   \tau^{(2n_k-1)/3}(ab)^{n_k} +    (ab)^{2n_k} \bigr)
%\tau^3 Z_k^4\\
%& \quad +(a+b)  (ab)^{2n_k}\tau^{(n_k+1)/3+2} Z_k^3 +(a+b)\tau^2  (ab)^{3n_k}Z_k^2\\
%& \quad + (b(a+b)\tau^2 + a^2)(ab)^{4n_k} \\
%& = 
\tau^4 \Delta_k Z_k^4 + [ (a+b)   \tau^{(2n_k-1)/3}(ab)^{n_k} +    (ab)^{2n_k} ]\tau^3 Z_k^4\\
& \quad +(a+b)  (ab)^{2n_k}\tau^{(n_k+7)/3} Z_k^3 +(a+b)\tau^2  (ab)^{3n_k}Z_k^2\\
& \quad + (b(a+b)\tau^2 + a^2)(ab)^{4n_k},
\end{align*}
where 
%\begin{align*}
$$
\Delta_k
=Z_k^4 +  \tau^{(2n_k -1)/3  }  Z_k^2 +(a+b)  \tau^{n_k -1} Z_k  +(b(a+b)\tau^2 + a^2)\tau^{(4n_k -8)/3}
%\end{align*}
$$
has been already introduced in the proof of Proposition \ref{polzk}. 

It follows from Proposition \ref{polzk} that
$$
\tau^4 \Delta_k + (a+b)   \tau^{(2n_k-1)/3}(ab)^{n_k} \tau^3 = \tau^4 (ab)^{2 n_k}.
$$
Consequently, we get
\begin{align*}
\delta_k & = 
\tau^4 (ab)^{2 n_k} Z_k^4 +   (ab)^{2n_k} \tau^3 Z_k^4\\
& \quad +(a+b)  (ab)^{2n_k}\tau^{(n_k+1)/3+2} Z_k^3 +(a+b)\tau^2  (ab)^{3n_k}Z_k^2\\
& \quad + (b(a+b)\tau^2 + a^2)(ab)^{4n_k}. 
\end{align*}
This shows that $\delta_k$ can be written as $(ab)^{2 n_k}$ times a polynomial in $a$ and $b$,
which depends on $k$. 
\end{proof}

%For the sake of clarity, 
To conclude the proof of Theorem \ref{mainth}, 
we explain the relationship between $P_k, Q_k$ and the convergents $p_\ell / q_\ell 
= p_\ell(a, b) / q_\ell (a,b)$ of the Thue--Morse continued fraction
$$
\xi_{a,b} = [0; a, b,b,a, b, a, a, b, \ldots].
$$
We have 
$$
q_{4^k} (a, b) = (a b)^{n_k} Q_k (a^{-1}, b^{-1}), \quad p_{4^k} (a, b) = (a b)^{n_k} P_k (a^{-1}, b^{-1}),
\quad k \ge 1, 
$$
and we check that 
$$
A_j = (ab)^4 a_j (a^{-1}, b^{-1}), \quad 0 \le j \le 4.
$$
%Put $C_j = (ab)^4 a_j (a^{-1}, b^{-1})$, we have
%\begin{align*}
%C_2&= a^2 b^2 (a^2 b^2 + a^2 + b^2) \\
%C_0&= b (a+b) C_2 + a^2 b^4\\
%C_1=C_3&=(a+b) ab C_2\\
%C_4&= a (a+b) C_2+ a^2 b^4,
%\end{align*}
%which are the same coefficients as in the Theorem, 
For $k \ge 1$, put
\begin{align*}
\eps_{4^k} = & A_4 \Bigl( \frac{p_{4^k}}{q_{4^k}} \Bigr)^4 + A_3 \Bigl( \frac{p_{4^k}}{q_{4^k}} \Bigr)^3
+ A_2  \Bigl( \frac{p_{4^k}}{q_{4^k}} \Bigr)^2   + A_1 \frac{p_{4^k}}{q_{4^k}} + A_0 \\
& = 
\frac{(ab)^{-2 n_k+4} T_k (a^{-1}, b^{-1})}{(ab)^{- 4 n_k} q_{4^k}^4} 
	= \frac{(ab)^{2 n_k+4} T_k (a^{-1}, b^{-1})}{q_{4^k}^4}.
\end{align*}
Set $d = \deg a + \deg b$. Since
$$
|q_{4^k}| = 2^{d n_k}, \quad |T_k (a^{-1}, b^{-1}) | \le 1, \quad 
|ab|^{2 n_k} = 2^{2 d n_k},
$$
we get
$$
|\eps_{4^k} | \le 2^{-2 d n_k} = |q_{4^k}|^{-2}. 
$$
Recalling that $|\xi_{a,b} - p_{4^k} / q_{4^k}| < |q_{4^k}|^{-2}$, we derive that    %%y  
$$
| A_4 \xi_{a, b}^4 + A_3 \xi_{a, b}^3 + A_2 \xi_{a, b}^2 + A_1 \xi_{a, b} + A_0 | 
\le \max\{|A_1|, \ldots , |A_4|\} \cdot 2^{-2 d n_k}. 
$$
Since $n_k$ is arbitrarily large, this gives that 
$$
A_4 \xi_{a, b}^4 + A_3 \xi_{a, b}^3 + A_2 \xi_{a, b}^2 + A_1 \xi_{a, b} + A_0 = 0,
$$
and concludes the proof of the theorem.

\section{Proof of Proposition \ref{hypric}}   \label{proofautres}

%\begin{proof}[Proof of Proposition \ref{hypric}]
%We show below that $\xi_{a, b}$ is differential-quadratic and 
%satisfies a simple Riccati equation.  
%Observe that 
%$$
%A_0 + A_2 + A_4  =(a+b+ab)^4.
%$$
%Since $(A_2)'=0$ (it is a square) and $(A_0 + A_2 + A_4)'= 0$, we get that $(A_0)'= (A_4)'$.
Since 
$$
A_0 + A_4 = (a + b)^2 (ab + a + b)^2
$$
is a square, its derivative is $0$ and we get that $(A_0)'= (A_4)'$.
By deriving the minimal defining polynomial $A_4 X^4 + \ldots + A_1 X + A_0$ of $\xi_{a,b}$, we obtain
$$
\sum_{j=0}^4 \, (A_j)' X^j + A_1  X'  (1+X^2) =0,
$$
hence
$$
(A_0)' (1+X^4) + (A_1)'  (X+X^3) = A_1 X' (1+ X^2) 
$$
and
$$
(A_0)' (1+X^2) + (A_1)' X = A_1 X',
$$
that is
$$ 
(A_0)' (1+ X^2) = (A_1 X)' .
$$
Since
$$
(A_0)'= (ab)' (a+b+ab)^2, \quad A_1= ab (a+b) (a+b+ab)^2,
$$
the last equation becomes
$$
[(ab (a+b)) X]'= (ab)' (1+ X^2).
$$
This establishes that $\xi_{a, b}$ is differential-quadratic and satisfies a simple Riccati equation.  

Assume now that there are an integer $s \ge 2$ and polynomials 
$A, B, C$ and $D$ in $\F_2[z]$ such that $\xi_{a, b}$ satisfies 
$$
A \xi_{a, b}^{2^s + 1} + B \xi_{a, b}^{2^s} + C \xi_{a, b} + D = 0.
$$
Then, the polynomial $A X^{2^s + 1} + B X^{2^s} + C X + D$ must be a multiple 
of the minimal defining polynomial $A_4 X^4 + \ldots + A_1 X + A_0$ of $\xi_{a,b}$ and there 
exist $c_0, c_1, \ldots , c_{2^s - 3}$ in $\F_2[z]$ such that 
$$
(A_4 X^4 + \ldots + A_1 X + A_0) \, (c_{2^s-3} X^{2^s - 3} + \ldots + c_1 X + c_0) =
A X^{2^s + 1} + B X^{2^s} + C X + D.
$$
We thus get $2^s - 2$ linear forms in $c_0, \ldots , c_{2^s - 3}$ which vanish. 
The associated matrix is a pentadiagonal (if $s \ge 3$, the case $s=2$ being immediate) 
Toeplitz matrix of the form
$$
M_s = 
\begin{pmatrix}
	A_2 &     A_1 &    A_0    & 0   & 0   & 0   & \cdots &   0 &   0 &   0  \\%  & 0 \\
	A_3 &    A_2 &    A_1   &    A_0 & 0  & 0 & \cdots & 0   & 0   & 0    \\% & 0 \\
	A_4 & A_3   & A_2    & A_1   & A_0  & 0 & \cdots & 0  & 0  & 0  \\%  & 0 \\
	0 &    A_4 & A_3      & A_2  & A_1  & A_0  & \cdots & 0 & 0 &  0 \\% & 0 \\
	%\vdots & \cdots & \cdots & \cdots & \cdots & \cdots & \cdots & \cdots & \cdots & \cdots  \\% & \cdots  \\
	\vdots & \vdots & \vdots & \vdots & \vdots & \vdots & \ddots & \vdots & \vdots & \vdots \\%  & \cdots  \\
	0 & 0 & 0 & 0 & 0 & 0 & \cdots  & A_4 & A_3 & A_2 \\%  & A_1 \\
%	0 & 0 & 0 & 0 & 0 & 0 & \cdots & 0 & A_4 & A_3  \\%  & A_2 \\  
\end{pmatrix} . 
$$
To compute its determinant, we simply expand it and observe that $A_2$ is the only coefficient 
of the minimal defining polynomial of $\xi_{a,b}$ which involves the term $a^4 b^4$. All the other terms
involved are of the form $a^i b^j$, with $0 \le i, j \le 4$ and $i+j \le 7$. 
Consequently, the determinant of $M_s$ is equal to $(a^4 b^4)^{2^s - 2}$ plus a linear combination 
of terms $a^i b^j$, with $0 \le i, j \le 4 (2^s - 2)$ and $(i, j) \not= (4 (2^s - 2), 4 (2^s - 2))$. 
In particular, it does not vanish and the system of equations has only the trivial solution 
$c_0 = \ldots = c_{2^s - 3} = 0$. This shows that $\xi_{a,b}$ cannot be 
hyperquadratic. 
%\end{proof}

\section{Concluding remarks}\label{conc}   

We gather in this concluding section several additional results, without their proofs.  %%y 

\subsection{Explicit expressions.} By Proposition \ref{reczk}, we can derive an explicit formula for $Z_k$, namely %%y 
$$
Z_{k} = \tau^{{(2^{2k-1}-2)/3} 
}  \Bigl(
1+b
+ \sum_{j=0 }^{2k-2}\tau^{({2-2^{j+1} -\chi(j)})/{3} 
	} (ab)^{2^{j}} 
	\Bigr), \quad k \ge 1, 
$$
where $\chi(j):= j \pmod 2$. Then, by Lemma \ref{ZPQR}, we obtain the following
explicit formula for $P_k/Q_k$.

\begin{proposition}\label{PkQk}
For $k \ge 2$, we have
$$
\frac {P_k}{Q_k} =
\frac{a+a^2b +ab^2 + 
(a+b) \alpha_k 
}{ 1+b^2+ a^2b^2 + \alpha_k + \alpha_k^2 },
$$
where
$$
\alpha_k =   \sum_{j=1 }^{k-1}\tau^{({2-2^{2j+1} })/{3} } (ab)^{2^{2j}} .
$$
\end{proposition}
Notice that $\alpha_k$ satisfies the relation 
$$
\alpha_k^4 + \tau^2 \alpha_k + 
\tau^{({8-2^{2k+1} })/{3} } (ab)^{2^{2k}} 
 + a^4b^4 =0.
$$

Proposition \ref{PkQk} implies the following theorem 
by taking the limit as  $k$ goes to infinite.

\begin{theorem}  \label{alpha}
Set
$$
\alpha =   \sum_{j=1 }^{\infty}\tau^{({2-2^{2j+1} })/{3} } (ab)^{2^{2j}}.
$$
Then, $\alpha$ is algebraic and it satisfies %%y 
$$
\alpha^4 + \tau^2 \alpha + 
  a^4b^4 =0.
$$
Furthermore, the Thue--Morse continued fraction $\xi_{a,b}$ can be expressed as  %%y 
$$
\xi_{a, b} =
\frac{a+a^2b +ab^2 + 
(a+b) \alpha 
}{ 1+b^2+ a^2b^2 + \alpha + \alpha^2 }.
$$
\end{theorem}

We stress that Theorem \ref{alpha} implies that $\xi_{a, b}$ is algebraic.  %%y 

\subsection{Even and odd sections.}
To understand the structure of $\xi_{a,b}$, we let
$$
\xi_{a,b} = \xi^{ee} + \xi^{eo} + \xi^{oe} + \xi^{oo},
$$
where
$$
\xi^{ee}= \sum a^{2i}b^{2j}, \quad
\xi^{eo}= \sum a^{2i}b^{2j+1}, \quad
\xi^{oe}= \sum a^{2i+1}b^{2j}, \quad
\xi^{oo}= \sum a^{2i+1}b^{2j+1}.
$$
We can derive, by Theorem \ref{mainth}, 
that
$$
\xi^{ee} = \xi^{oo} =0,
$$
and
$\xi^{eo}, \xi^{oe}$ are algebraic.
The coefficients of $\xi_{a,b}$ are reproduced in Figure 1 in the following manner.
Let $S=\{(-1,0), (-2, -1), (-2, -3), \ldots\}$ be the set of the blue dots in Figure 1. Then
$$
\xi_{a,b} = \sum_{(i,j)\in S} a^i b^j = a^{-1} + a^{-2}b^{-1} + a^{-2} b^{-3} + \cdots
$$
\begin{figure}
  \centering
  \includegraphics[width=0.9\textwidth]{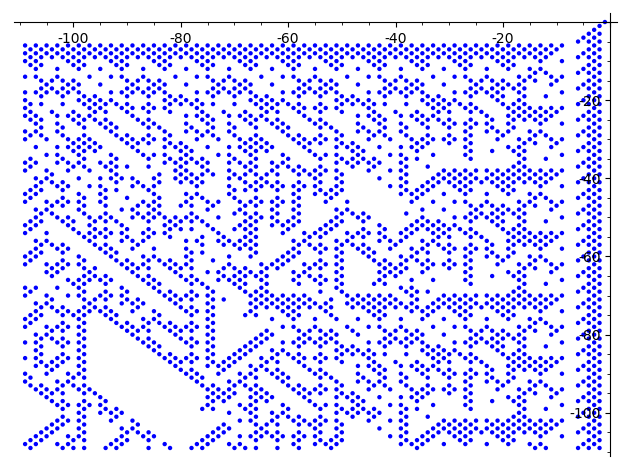}
  \caption{Coefficients of $\xi_{a,b}$}
  \label{figure_F1}
\end{figure}

\subsection{Jacobi continued fraction for formal power series}
In the field of formal power series $\F((x))$ over a field $\F$, %%y 
the Jacobi continued fraction $J(\mathbf{u},\mathbf{v})$ defined by  %%y
two sequences $\mathbf{u}=(u_n)_{n \ge 1}$ and $\mathbf{v}=(v_n)_{n \ge 0}$ with $v_n \neq 0$ for all  %%y
$n \ge 0$ is the infinite continued fraction %%y 
\begin{equation*}
J(\mathbf{u},\mathbf{v})
=
\cfrac {v_0 }{1 + u_1 x - 
  \cfrac{v_1 x^2 }{ 1+u_2x - 
    \cfrac{v_2 x^2 }{ {1 + u_3x - 
\cfrac{v_3 x^2}{ \ddots}}} }}. 
\end{equation*}
The basic properties of Jacobi continued fractions can be found in
\cite{Flajolet1980, Wall1948}.   %%y  
The Hankel determinants  $H_n(J(\mathbf{u},\mathbf{v})) $ %%y 
of $J(\mathbf{u},\mathbf{v})$ can be calculated
by means of the following  fundamental relation, first
stated by Heilermann in 1846 \cite{Heilermann1846}:
$$
H_n(J(\mathbf{u},\mathbf{v})) 
		= v_0^n v_1^{n-1} v_2^{n-2} \cdots v_{n-2}^2 v_{n-1}. 
$$
Assume that $\F$ is the field $\F_2$.  If the Jacobi continued fraction exists, then    %%y 
$v_j=1$ for $j\geq 0$. In this case, the Hankel determinants $H_n$ are all equal to $1$.  %%y 
By \cite{Guo2021HW, Allouche2020HN}, the sequence $(c_n)_{n\geq 0}$ of the 
coefficients of $J(\mathbf{u},\mathbf{v})=\sum_{n\geq 0} c_n x^n $
is {\it apwenian}. This means that the following relations hold:
$$
c_0 = 1  \ \text{and}  \quad c_n \equiv c_{2n+1} + c_{2n+2} \ \bmod 2, \quad n \ge 0.     %%y 
$$
Consider the Jacobi continued fraction $\omega (x)$ defined by the Thue--Morse sequence %%y 
\begin{equation*}
\omega(x)
=
\cfrac {1 }{1 + u_1 x + 
  \cfrac{ x^2 }{ 1+u_2x + 
    \cfrac{x^2 }{ {1 + u_3x + 
\cfrac{x^2}{ \ddots}}} }}, 
\end{equation*}
where $(u_1, u_2, \ldots) = (1,0,0,1,0,1,1,0, \ldots)$ is the Thue--Morse
sequence over $\{0, 1\}$. 
Since
$$
\frac{\omega(z^{-1})}{z} = \xi_{z+1,z},
$$
it follows from Theorem \ref{mainth} that %%y 
$$
g_4\, \omega(x)^4 + g_3\, \omega(x)^3 + g_2\, \omega(x)^2 + g_1\, \omega(x) + g_0=0,
$$
where
\begin{align*}
	g_0 &= x^5 + x^3 + x^2 + x + 1, \\
	g_1 &= x^6 + x^5 + x^4 + x^3 + x^2 + x, \\
	g_2 &= x^6 + 1, \\
	g_3 &= x^8 + x^7 + x^6 + x^5 + x^4 + x^3, \\
	g_4 &= x^{10} + x^9 + x^8 + x^7 + x^5 + x^4.
\end{align*}
Consequently, $\omega (x)$ is algebraic of degree $4$. %%y 

\subsection{Thue--Morse continued fractions in the ring of formal power series.}
Let $\K$ be a {\it ring} of characteristic $2$. Consider the Thue--Morse 
continued fractions $\eta$ in the ring of formal power series $\K[[a,b]]$ defined by 
\begin{equation}\label{eta}
	\eta:=\eta_{a,b}	:=\cfrac {a}{ 1+ \cfrac{b}{ 1+ \cfrac{b}{ 1+\cfrac {a}{\ddots}}} }, 
\end{equation}
where the sequence of partial numerators $(a,b,b,a,\ldots)$ is the Thue--Morse sequence over $\{a, b\}$. 
By using the same method as in Section \ref{proofmain} for the proof of Theorem \ref{mainth}, 
we can establish the following result.

\begin{theorem}\label{ring}
	The Thue--Morse continued fraction $\eta$ defined in \eqref{eta} is algebraic of degree $4$. More precisely, we have
$$ a(a+b)^3 + a{b}+  (a+b)(a + b + 1)\eta+(a+b + 1)\eta^2+\eta^4=0.$$
\end{theorem}

Working in $\K[[a,b,x]]$ and setting $a:=ax$ and $b:=bx$,  Theorem \ref{ring}
implies that Conjecture 1.3 in \cite{HuHan} is true.

\section*{Acknowledgement}    
The authors thank Alain Lasjaunias for useful discussion.

%\vfill\eject

\end{document}